\documentclass[11pt]{amsart}
\usepackage[top=30truemm,bottom=30truemm,left=30truemm,right=30truemm]{geometry} 
\usepackage{amssymb, array}
\usepackage[all]{xy}
\usepackage{ascmac}
\usepackage[dvips]{graphicx}
\usepackage{color}
\usepackage{amscd}
\usepackage{fancyhdr}
\usepackage[driverfallback=dvipdfm]{hyperref}
\usepackage{amsrefs}
\usepackage{cleveref}
\usepackage{listings}
\usepackage{amsmath}
\usepackage{here}

\title{How to calculate the proportion of everywhere locally soluble diagonal hypersurfaces}
\author{Yoshinori Kanamura}
\address{Y. Kanamura \\ Department of Mathematics \\ Faculty of Science and Technology, Keio University, 3-14-1, Hiyoshi, Kohoku, Yokohama, Kanagawa, Japan}
\email{kana1118yoshi@keio.jp}

\author{Yoshinosuke Hirakawa} 
\address{Y. Hirakawa \\ Department of Mathematics \\ Faculty of Science and Technology \\ Keio University, 3-14-1, Hiyoshi, Kohoku, Yokohama, Kanagawa, Japan}
\email{hirakawa@keio.jp}
	\thanks{This research was supported by the Research Grant of Keio Leading-edge Laboratory of Science \& Technology (Grant Number  2019-2020 000074).
	This research was supported in part by KAKENHI 18H05233.
	This research was conducted as part of the KiPAS program FY2014--2018 of the Faculty of Science and Technology at Keio University.}

\subjclass[2010]{primary 	14G05;  
	secondary
		11D09; 
		11D25; 
		11D41; 
		11D72; 
		11E76; 
		}
\keywords{Diophantine equations,
	Hasse principle,
	rational points}
\date{\today}
 
\theoremstyle{plain}
 \newtheorem{theorem}{Theorem}[section]
 \crefname{theorem}{Theorem}{Theorems}
 \newtheorem{proposition}[theorem]{Proposition}
 \crefname{proposition}{Proposition}{Propositions}
 \newtheorem{lemma}[theorem]{Lemma}
 \crefname{lemma}{Lemma}{Lemmas}
 
 \crefname{corollary}{Corollary}{Corollaries}
 
 \crefname{conjecture}{Conjecture}{Conjectures}

 \crefname{question}{Question}{Questions}
 
 \crefname{problem}{Problem}{Problems}

\theoremstyle{definition} 
 \newtheorem{definition}[theorem]{Definition}
 \crefname{definition}{Definition}{Definitions}
 
 \crefname{example}{Example}{Examples}
 \newtheorem{remark}[theorem]{Remark}
 \crefname{remark}{Remark}{Remarks}
 
 \crefname{caution}{Caution}{Cautions}
 
 \crefname{equation}{formula}{formulas}

\newcommand{\pintp}{\Z_p}

\newcommand{\pintt}[1]{\Z_{#1}^{\times}}

\newcommand{\Proj}[1]{\mathbb{P}^{#1}}

\newcommand{\abs}[1]{\lvert#1\rvert}

\newcommand{\setmid}{\mathrel{}\middle|\mathrel{}} 

\newcommand{\A}{\mathbb{A}}

\newcommand{\F}{\mathbb{F}}

\newcommand{\PP}{\mathbb{P}}
\newcommand{\Q}{\mathbb{Q}}
\newcommand{\R}{\mathbb{R}}
\newcommand{\Z}{\mathbb{Z}}

\newcommand{\loc}{\mathrm{loc}}

\newcommand{\bb}[1]{\boldsymbol{#1}}
\newcommand{\ba}{\boldsymbol{a}}

\DeclareMathOperator{\image}{Im} 
\DeclareMathOperator{\wt}{\boldsymbol{w}} 

\begin{document}


\maketitle

\begin{abstract}
In this paper,
we establish a strategy for the calculation of the proportion of everywhere locally soluble diagonal hypersurfaces of $\mathbb{P}^{n}$ of fixed degree.
Our strategy is based on the product formula established by Bright, Browning and Loughran. Their formula reduces the problem into the calculation of the proportions of $\mathbb{Q}_{v}$-soluble diagonal hypersurfaces for all places $v$.
As worked examples,
we carry out our strategy in the cases of quadratic and cubic hypersurfaces.
As a consequence, 
we prove that around $99.99\%$ of diagonal cubic $4$-folds have $\mathbb{Q}$-rational points under a hypothesis on the Brauer-Manin obstruction. 
\end{abstract}

\section{Introduction}

In arithmetic geometry, 
for a given family of algebraic varieties,
it is natural to ask how often its member has a rational point.
Poonen and Voloch \cite{Poonen-Voloch} gave a philosophy that,
for a nice family, the proportion of members which have $\Q$-rational points can be represented by the product of the proportions of those which have $\Q_{v}$-rational points for every place $v$.
Following their philosophy,
the proportions have been calculated for some special families in e.g.\ \cite{BBL, BCFJK, Browning, Browning_cubic} under conditions on the Brauer-Manin obstruction. 
There are also other related works e.g.\ \cite{Bhargava, BGW, Mitankin-Salgado, Poonen-Stoll}.

Among algebraic varieties,
diagonal hypersurfaces of $\PP^{n}$ have attracted special interest because of their remarkable arithmetic properties. 
For example,
although cubic hypersurfaces of $\PP^{6}$ may not have $\Q_{v}$-rational points in general,
Lewis \cite{Lewis_cubic} proved that diagonal cubic hypersurfaces of $\PP^{6}$ have $\Q_{v}$-rational points for every place $v$.
Moreover,
although cubic hypersurfaces of $\PP^{8}$ may not have $\Q$-rational points in general,
Hooley \cite{Hooley_I} proved that diagonal cubic hypersurfaces of $\PP^{8}$  have $\Q$-rational points by establishing the Hasse principle for non-singular cubic hypersurfaces of $\PP^{8}$ (see also \cite{Hooley_RH}).

In this paper,
we study the proportion of diagonal hypersurfaces of $\PP^{n}$ which have $\Q$-rational points by following the above philosophy of Poonen and Voloch. 
Before our main results,
let us explain some simple applications.
Fix $n, k \in \Z_{\geq 2}$. For each $\bb{a} = (a_{0}, \dots, a_{n}) \in \Z^{\oplus n+1}$, let $X_{\bb{a}}^{k}$ be a diagonal hypersurfaces of $\PP^{n}$ defined by $\sum_{i = 0}^{n} a_{i}x_{i}^{k} = 0$, and set $\abs{\bb{a}}$ as 
$\max\{ \abs{a_{0}}, \abs{a_{1}},\cdots, \abs{a_{n}} \}$
with the Euclidean norm $\abs{\cdot}$ on $\R$. 
Moreover, set 
\[
\rho(n, k)
:= \lim_{H \to \infty} \frac{\left\{ \bb{a} \in \Z^{\oplus n+1} \setmid
	\text{$\lvert \bb{a} \rvert < H$ and $X_{\bb{a}}^{k}(\Q) \neq \emptyset$} \right\}}
{\left\{ \bb{a} \in \Z^{\oplus n+1} \setmid \lvert \bb{a} \rvert < H \right\}}
\]
if the limit exists.

\begin{theorem}\label{applicaiton}
\begin{enumerate}
	\item 
	For $k = 2$, we have the following table.
	\begin{table}[H]
		\begin{tabular}{|c|c|c|c|c|c|} \hline
			$n$ & $2$ & $3$ & $\geq 4$ \\ \hline
			$\rho(n, 2)$ & $0$ & $0.8268\ldots$ & $1-2^{-n}$ \\ \hline
		\end{tabular}
	\end{table}
		
	\item 
	For $k = 3$, we have the following table under the assumption that if $3 \leq n \leq 7$, then the Brauer-Manin obstruction is the only obstruction to the Hasse principle for diagonal cubic $(n-1)$-folds.
	\begin{table}[H]
		\begin{tabular}{|c|c|c|c|c|c|} \hline
			$n$ & $2$ & $3$ & $4$ & $5$ & $\geq 6$ \\ \hline
			$\rho(n, 3)$ & $0$ & $0.8964\ldots$ & $0.9965\ldots$ & $0.9999\ldots$ & $1$ \\ \hline
		\end{tabular}
	\end{table}
\end{enumerate}
\end{theorem}

For the proof of \cref{applicaiton},
see \S3.3.
Note that there is no known cubic $(n-1)$-fold ($n \geq 4$) which violates the Hasse principle
(cf. \cite[p.49]{Colliot-Thelene-Swinnerton-Dyer}, \cite[Conjecture 3.2 and Appendix A]{Poonen-Voloch}).
For the detail on the Brauer-Manin obstruction,
see e.g.\ \cite{Manin_2nd,Poonen}.

In fact, some values in the above tables have been already known in the literature 
(e.g. \cite{BBL,Serre_course,Browning-Dietmann,Lewis_cubic}),
but the authors could not find any explicit references which contain the values $\rho(3,2)$, $\rho(4,3)$, $\rho(5,3)$.

Thanks to \cite[Theorem 1.4]{Browning},
under the condition in \cref{applicaiton}, the value $\rho(n, k)$ coincides with its local avatar $\rho_{\loc}(n, k)$, which is defined by
\[
\rho_{\loc}(n, k) := \lim_{H \to \infty}
\frac{\left\{ \bb{a} \in \mathbb{Z}^{\oplus n+1} \setmid
	\text{$\abs{\bb{a}} < H$ and $X_{\bb{a}}^{k}(\Q_{v}) \neq \emptyset$ for every place $v$}  \right\}}{\#\left\{ \bb{a} \in \mathbb{Z}^{\oplus n+1} \setmid \abs{\bb{a}} < H \right\}} 
\]
if the limit exists.
For a sufficient condition so that $\rho(n, k) = \rho_{\loc}(n, k)$ in the case $k \geq 4$ due to \cite{Brudem-Dietmann}, 
see \cref{the case k>3}. 
Moreover, set 
\begin{align*}
\rho_{v}(n, k) :=\begin{cases}
\mu_{p}\left( \left\{ \bb{a} \in \pintp^{\oplus n+1}
\setmid \text{$X_{\bb{a}}^{k}(\Q_p) \neq \emptyset$} \right\} \right) &\text{if $v$ is a prime $p$},\\
2^{-n-1}\mu_{\infty}\left(\left\{\bb{a}\in [-1, 1]^{n+1} \setmid X_{\bb{a}}^{k}(\R)\neq \emptyset\right\}\right) &\text{if $v = \infty$}.
\end{cases}
\end{align*}
Here, $\mu_{p}$ is the Haar measure on $\Z_{p}$ normalized so that $\mu_{p}(\Z_{p}) = 1$ and $\mu_{\infty}$ is the Lebesgue measure on $\R$.
We use the same letter $\mu_{p}$ (resp.\ $\mu_{\infty}$) also for the product measure on $\Z_{p}^{\oplus n+1}$ (resp.\ $\R^{\oplus n+1}$).
The calculation of $\rho_{\loc}(n, k)$ is reduced to that of $\rho_{v}(n, k)$ for every place $v$ by the following theorem of Bright, Browning and Loughran.

\begin{theorem}[{a special case of \cite[Theorem 1.3]{BBL}}]\label{thm: localprop} 
	The limit $\rho_{\loc}(n, k)$ exists
	\footnote{
	Note that $\rho_{\loc}(2, k) = 0$ for every $k \geq 2$
	as proven in \cite[Theorem 1.1]{Browning-Dietmann}.
	On the other hand, \cref{generic} implies that
	$\prod_{v : \text{place}} \rho_{v}(2, k) = 0$
	in the sense that
	$\sum_{v < H} \log\rho_{v}(2, k) \to -\infty$ ($H \to \infty$)
	because $\sum_{p < H, p \equiv 1 \bmod{k}} p^{-1} \to \infty$ ($H \to \infty$)
	(cf. \cite[p75]{Serre_course}).   
	}
	and is given by
	\begin{align*}
	\rho_{\loc}(n, k) = \prod_{v : \text{place}} \rho_{v}(n, k).
	\end{align*}
	Moreover,
	$\rho_{loc}(n, k)$ is positive whenever $n \geq 3$.
\end{theorem}

In this paper, we establish a strategy to calculate $\rho_{v}(n, k)$ for all $v$ for each fixed $n$ and $k$ (see \cref{pathologicalprime}).
Our strategy is regarded as a quantitative refinement of the argument in \cite[\S3]{Browning-Dietmann}.
As worked examples,
we carry out our strategy in the cases $k =2, 3$. 
In particular, 
\cref{applicaiton} follows immediately from  the following \cref{quadratic_main,cubic_main}.

\begin{theorem} [$k = 2$] \label{quadratic_main}
Suppose that $k = 2$. 
	\begin{enumerate}
		\item
		If $n = 2$, then
		\[
		\rho_{p}(2, 2)
		= \begin{cases}
		\displaystyle \frac{7}{12} &\text{if $p = 2$}, \\
		\displaystyle 1 - \frac{3}{2}p^{-1} \left( \frac{1-p^{-1}}{1-p^{-2}} \right)^{2} & \text{otherwise}.
		\end{cases}
		\]

		\item
		If $n = 3$, then
		\[
		\rho_{p}(3, 2)
		= \begin{cases}
		\displaystyle \frac{1231}{1296} &\text{if $p = 2$}, \\
		\displaystyle 1 - \frac{3}{2}p^{-2} \left( \frac{1-p^{-1}}{1-p^{-2}} \right)^{4} & \text{otherwise}.
		\end{cases}
		\]
		
		\item $($cf.\ \cite[Corollary 2]{Serre_course}$)$
		If $n \geq 4$,
		then $\rho_{p}(n, 2) = 1$.
	\end{enumerate}
\end{theorem}

Note that the value $\rho_{\infty}(n, 2)$ equals $1-2^{-n}$ for every $n \in \Z_{\geq 2}$.

\begin{theorem} [$k = 3$] \label{cubic_main}
	Suppose that $k = 3$. 
	\begin{enumerate}	
		\item 
		If $n = 2$, then
		\[
		\rho_p(2, 3)
		= \begin{cases}
		\displaystyle \frac{13831}{19773} &\text{if} \ p=3, \\
		\displaystyle 1 - 2p^{-1}\left(\frac{1 - p^{-1}}{1 - p^{-3}}\right) &\text{if $p \equiv 1 \bmod 3$}, \\ 
		\displaystyle 1 - 6p^{-3}\left(\frac{1-p^{-1}}{1-p^{-3}}\right)^3 &\text{if $p\equiv 2 \bmod 3$}. 
		\end{cases}
		\]
		
		\item $($= \cite[Theorem 2.2]{BBL}$)$
		If $n = 3$, then
		\[
		\rho_p(3, 3)
		= \begin{cases}
		\displaystyle \frac{6391}{6591} &\text{if $p = 3$},\\
		\displaystyle 1 - \frac{8}{3} p^{-2}(1+p^{-1})^{2} \left(\frac{1-p^{-1}}{1-p^{-3}}\right)^3
		&\text{if $p \equiv 1 \bmod{3}$}, \\ 
		\displaystyle 1 &\text{if $p \equiv 2 \bmod{3}$}.
		\end{cases}
		\]
		
		\item
		If $n = 4$, then
		\[
		\rho_p(4, 3)
		= \begin{cases}
		\displaystyle 1 - \frac{40}{3}p^{-4} \left( \frac{1-p^{-1}}{1-p^{-3}} \right)^{4} &\text{if $p \equiv 1 \bmod 3$}, \\
		1  &\text{otherwise}. 
		\end{cases}
		\]
		
		\item
		If $n = 5$, then
		\[
		\rho_p(5, 3)
		= \begin{cases}
		\displaystyle 1- \frac{80}{3}p^{-6} \left( \frac{1-p^{-1}}{1-p^{-3}} \right)^{6} &\text{if $p \equiv 1 \bmod 3$}, \\
		1  &\text{otherwise}. 
		\end{cases}
		\]
		
		\item $($cf.\ \cite[Theorem 2]{Lewis_cubic}$)$
		If $n\geq 6$, 
		then $\rho_p(n, 3) = 1$.
	\end{enumerate}
\end{theorem}

Note that the value $\rho_{\infty}(n, 3)$ equals $1$ for every $n \in \Z_{\geq 2}$.

The plan of this paper is as follows. 
In \S2 we give a general upper bound for the value $\rho_{p}(n, k)$. 
It is sufficient for determination of the value $\rho_{p}(n, k)$ for a generic prime. 
In \S3 we calculate the values $\rho_{p}(n, k)$ for \textit{pathological} pairs $(p, k) = (2,2), (3,3)$ and complete the proofs of \cref{quadratic_main,cubic_main} hence of \cref{applicaiton}. 
Our strategy in \S3 works also for general $n$ and $k$ in principle. 
In \S4 we discuss a consequence for the proportions of (uni)rationality.

\subsection*{Notation}
For each prime $p$,
let $\Z_{p}$ be the ring of $p$-adic integers
and $\Q_{p}$ be its field of fractions.
We denote the (additive) $p$-adic valuation map by $v_{p} : \Q_{p}^{\times} \to \Z$,
and we use the same symbol also for its direct sum $v_{p} : (\Q_{p}^{\times})^{\oplus n+1} \to \Z^{\oplus n+1}$
defined by $v_{p}(\bb{a}) := (v_{p}(a_{0}), \dots, v_{p}(a_{n}))$ for every $\bb{a} = (a_{0}, \dots, a_{n}) \in (\Q_{p}^{\times})^{\oplus n+1}$.

\section{$\rho_{p}(n, k)$ for generic primes}

For every $k, r \in \Z_{\geq 0}$,
set $[k] := \{ 0, 1, \dots, k-1 \} \subset \Z$,
and let $[k]^{(r)}$ be the set of subsets of $[k]$ consisting of $r$ elements.
For every $K = \{ k_{1}, \dots, k_{d} \} \subset [k]$,
set $\wt(K) := k_{1}+ \dots +k_{d}$.

In a similar manner to \cite[\S 2.1.1]{BBL}, 
we define an equivalence relation $\simeq$ on $\Q_{p}^{\oplus n+1}$ as follows. 
Set $\Gamma_{p}(n, k) := \Q_{p}^{\times} \times
	\left( (\Q_{p}^{\times k})^{\oplus n+1} \rtimes \mathfrak{S}_{n+1} \right)$.
Here, the semi-direct product $(\Q_{p}^{\times k})^{\oplus n+1} \rtimes \mathfrak{S}_{n+1}$
is defined by the natural left permutation action of $\mathfrak{S}_{n+1}$.
Define an action of $\Gamma_{p}(n, k)$ on $\Q_{p}^{\oplus n+1}$ by
\[
	(\alpha; \alpha_{0}^{k}, \dots, \alpha_{n}^{k}; \sigma) \cdot (a_{0}, \dots, a_{n})
	:= (\alpha a_{\sigma(0)}(\alpha_{0})^{k}, \dots, \alpha a_{\sigma(n)}(\alpha_{n})^{k}).
\]
Define an equivalence relation $\simeq$ on $\Q_{p}^{\oplus n+1}$ by
$\bb{a} \simeq \bb{b}$ if there exists $\gamma \in \Gamma_{p}(n, k)$ such that $\bb{a} = \gamma(\bb{b})$.
Then, $X_{\bb{a}}^{k}$ is isomorphic to $X_{\bb{b}}^{k}$ over $\Q_{p}$ if $\bb{a} \simeq \bb{b}$. 
Moreover,
since the set $\Q_{p}^{\oplus n+1} / \simeq$ is finite,
we can calculate $\rho_{p}(n, k)$ for each fixed $n, k, p$ at least by tour de force analysis. 
In fact, by using the following \cref{generic}, 
we can calculate $\rho_{p}(n, k)$ for many $n, k, p$ uniformly.   

In what follows,
the measure $\mu_{p}(\left\{ \bb{a} \in \Z_{p}^{\oplus n+1} \setmid \cdots \right\})$  is abbreviated by $\mu_{p}(\cdots)$.
We set
\[
\kappa_{p}(n, k)
:= \frac{\mu_{p}\left( v_{p}(\bb{a}) \equiv \bb{0} \pmod{k} \right)}{\mu_{p}\left( v_{p}(\bb{a}) = \bb{0} \right)}
= \left( \sum_{e_{0}, \dots, e_{n} \geq 0} p^{-ke_{0}- \dots -ke_{n}} \right)
= (1-p^{-k})^{-n-1}.
\]
Intuitively, this quantity gives the ``expansion ratio by $(p^{k\Z_{\geq 0}})^{\oplus n+1}$-action". 

Let $\Delta_{n}$ be the set $\left\{\bb{a}\in\Z_{p}^{\oplus n+1} \setmid \prod_{i = 0}^{n}a_{i} = 0\right\}$. 
Then we have
\[
\mu_{p}(\text{$X_{\bb{a}}^{k}$ is singular})
= \mu_{p}\left(\Delta_{n} \right)
= 0.
\] 
Therefore, it is sufficient to consider $\bb{a}\in\Z^{\oplus n+1}\setminus\Delta_{n}$,
which are classified (not exculsively) into three types as follows.

\begin{definition}
	Let $n, k \in \Z_{\geq 2}$,
	$p$ be a prime,
	and $\bb{a} \in \Z_{p}^{\oplus n+1} \setminus \Delta_{n}$.
	
	\begin{enumerate}
		\item
		If there exist some $u_{0}, \dots, u_{n} \in \Z_{p}^{\times}$
		and $k_{3}, ..., k_{n} \in [k]$
		such that
		\[
		\bb{a} \simeq (u_{0}, u_{1}, u_{2}, p^{k_{3}}u_{3}, \dots, p^{k_{n}}u_{n}),
		\]
		then we say that $\bb{a}$ is of type I.
		
		\item
		If there exist some $u_{1}, \dots, u_{n} \in \Z_{p}^{\times}$, $t \in \Z_{p}^{\times k}$,
		and $k_{2}, ..., k_{n} \in [k]$
		such that
		\[
		\bb{a} \simeq (u_{1}, -u_{1}t, p^{k_{2}}u_{2}, p^{k_{3}}u_{3}, \dots, p^{k_{n}}u_{n}),
		\]
		then we say that $\bb{a}$ is of type II.
		
		\item
		If there exist some
		$r \in \Z_{\geq 0}$,
		$u_{1}, \dots, u_{n+1-r} \in \Z_{p}^{\times}$,
		$t_{1}, \dots, t_{r} \in \Z_{p}^{\times} \setminus \Z_{p}^{\times k}$,
		and distinct $k_{1}, ..., k_{n+1-r} \in [k]$
		such that
		\begin{equation} \label{insoluble}
		\bb{a} \simeq (p^{k_{1}}u_{1}, -p^{k_{1}}u_{1}t_{1}, \dots, p^{k_{r}}u_{r}, -p^{k_{r}}u_{r}t_{r},
		p^{k_{r+1}}u_{r+1}, \dots, p^{k_{n+1-r}}u_{n+1-r}),
		\tag{$\ast$}
		\end{equation}
		then we say that $\bb{a}$ is of type III.
	\end{enumerate}
\end{definition}

Moreover, by the Hasse-Weil bound for
non-singular projective curves defined by $u_{0}x_{0}^{k}+u_{1}x_{1}^{k}+u_{2}x_{2}^{k} = 0$ ($u_{0}, u_{1}, u_{2} \in \Z_{p}^{\times}$),
we obtain the following.

\begin{lemma} \label{generic_criterion}
	Let $n, k \in \Z_{\geq 2}$,
	$p$ be a prime such that $\gcd(p, k) = 1$,
	and $\bb{a} \in \Z_{p}^{\oplus n+1} \setminus \Delta_{n}$.
	Then, the following statements hold.
	
	\begin{enumerate}
		\item
		Suppose that $p \geq (k-1)(k-2)$ or $\gcd(p-1, k) = 1$.
		If $\bb{a}$ is of type I,
		then $X_{\bb{a}}^{k}(\Q_{p}) \neq \emptyset$.
		
		\item
		If $\bb{a}$ is of type II,
		then $X_{\bb{a}}^{k}(\Q_{p}) \neq \emptyset$.
		
		\item
		If $\bb{a}$ is of type III,
		then $X_{\bb{a}}^{k}(\Q_{p}) = \emptyset$.
	\end{enumerate}
\end{lemma}

\begin{proposition} \label{generic}
	Let $n, k \in \Z_{\geq 2}$,
	and $p$ be a prime.
	Suppose that $\gcd(p, k) = 1$.
	Then, we have
	\begin{align*}
	\rho_{p}(n, k)
	&\leq 1 - (n+1)! \left( \frac{1-p^{-1}}{1-p^{-k}} \right)^{n+1}
	\sum_{r \geq 0} \left( \frac{1}{2} - \frac{1}{2\gcd(p-1, k)} \right)^{r}
	\sum_{\substack{K \subset [k]^{(r)} \\ L \subset [k]^{(n+1-2r)} \\ \text{s.t. $K \cap L = \emptyset$}}}
	p^{-2\wt(K)-\wt(L)}
	\end{align*}
	Moreover,
	if $p \geq (k-1)(k-2)$ or $\gcd(p-1, k) = 1$,
	then the equality holds.
	\footnote{
	In fact, if one of the following conditions holds, 
	then $X_{\bb{a}}^{k}(\Q_p)\neq\emptyset$ for all $\bb{a}\in\Z^{\oplus n+1}\setminus \Delta_{n}$:
	\begin{enumerate}
		\item $p \geq (k-1)(k-2)$ and $n \geq 2k$. 
		\item $\gcd(p-1, k) = 1$ and $n \geq k$. 
	\end{enumerate}
	In particular, if (1) or (2) holds, then  we obtain $\rho_{p}(n, k) = 1$.
	}
\end{proposition}

Here, the sum with respect to $r$ is finite which runs over
$\max\{ n-k+1, 0 \} \leq r \leq \min\{ [\frac{n+1}{2}], k \}$,
where $[x]$ denotes the maximal integer not exceeding $x$.

\begin{proof}
	The whole statement is a direct consequence of \cref{generic_criterion}.
	Indeed, since the $\mathfrak{S}_{n+1}$-orbit of 
	$(p^{k_{1}}, p^{k_{1}}, \dots, p^{k_{r}}, p^{k_{r}}, p^{k_{r+1}}, p^{k_{r+2}}, \dots, p^{k_{n+1-r}})$
	with distinct $k_{i} \in [k]$
	consists of $(n+1)!/2^{r}$ vectors, 
	we obtain
	\begin{align*}
	\rho_{p}(n, k)
	&\leq 1 - \mu_{p}\left( \text{$\bb{a}$ is of type III} \right) \\
	&= 1 - \sum_{r \geq 0} \sum_{\{ k_{1}, \dots, k_{r} \} \subset [k]}
	\sum_{\substack{\{ k_{r+1}, \dots, k_{n+1-r} \} \\ \subset [k] \setminus \{ k_{1}, \dots, k_{r} \}}}
	\mu_{p}\left( \text{$\bb{a}$ satisfies \cref{insoluble} with some $u_{i}, t_{i}$} \right) \\
	&= 1 - \sum_{r \geq 0}
	\sum_{\substack{K = \{ k_{1}, \dots, k_{r} \} \subset [k]^{(r)} \\
			L = \{ k_{r+1}, \dots, k_{n+1-r} \}\subset [k]^{(n+1-2r)} \\ \text{s.t. $K \cap L = \emptyset$}}}
	\frac{(n+1)!}{2^{r}} p^{-2k_{1}- \dots -2k_{r}-k_{r+1}- \dots -k_{n+1-r}} \\
	&\qquad \times \kappa_{p}(n, k) \cdot
	\mu_{p}\left( \text{$\bb{a} = (u_{1}, -u_{1}t_{1}, \dots, u_{r}, -u_{r}t_{r}, u_{r+1}, \dots, u_{n+1-r})$
		with some $u_{i}, t_{i}$} \right) \\
	&= 1 - (n+1)! \sum_{r \geq 0} \frac{1}{2^{r}} 
	\sum_{\substack{K \subset [k]^{(r)} \\ L \subset [k]^{(n+1-2r)} \\ \text{s.t. $K \cap L = \emptyset$}}}
	p^{-2\wt(K)-\wt(L)} \\
	&\qquad \times (1-p^{-k})^{-(n+1)} \left( 1 - \frac{1}{\gcd(p-1, k)} \right)^{r} (1-p^{-1})^{n+1}
	\end{align*}
	as desired.
\end{proof}

\begin{remark}\label{pathologicalprime}
Thanks to \cref{generic},
in order to determine $\rho_{p}(n, k)$ for all primes $p$,
it is sufficient to consider the following two kinds of \textit{pathological} primes:
\begin{enumerate}
	\item
	$\gcd(p, k) \neq 1$.
	
	\item
	$p < (k-1)(k-2)$ and $\gcd(p-1, k) \neq 1$.
\end{enumerate}
The problems of these cases are as follows: 
\begin{itemize}
	\item In the case (1), every $\F_{p}$-rational point on a variety $X_{\bb{a}}^{k} \bmod{p}$ is singular. 
	\item In the case (2), a scheme $X_{\bb{a}}^{k} \bmod{p}$ may not have a $\F_{p}$-rational point. 
\end{itemize}
Anyway, since the number of pathological primes $p$ for each fixed $n$ and $k$ is finite,
we can determine $\rho_{p}(n, k)$ for all primes $p$ by tour de force
and eventually obtain $\rho_{\loc}(n, k)$.
\end{remark}

\section{$\rho_{p}(n, k)$ for pathological primes}

In this section,
we carry out tour de force analysis in order to calculate $\rho_{p}(n, k)$ for pathological primes $p$. 
Although we consider only the cases $k = 2$ and $k =3$, 
the same method also works for any $k$ in principle. 

If $k = 2$ (resp.\ 3),
then there is no prime number of second kind in \cref{pathologicalprime}.
Therefore,
it is sufficient to calculate $\rho_{2}(n, 2)$ (resp.\ $\rho_{3}(n, 3)$)
as we will do in what follows.

\subsection{The case of $k = 2$ and $p = 2$}

\begin{proposition} [$p = 2$ and $n = 2$] \label{quad_curve}
Let $u_{0}, u_{1}, u_{2} \in \Z_{2}^{\times}$.
\begin{enumerate}
\item
$X_{(u_{0}, u_{1}, u_{2})}^{2}$ has a $\Q_{2}$-rational point
if and only if
\[
	(u_{0}, u_{1}, u_{2}) \simeq (1, 1, 3), (1, 1, 7), (1, 3, 7).
\]

\item
$X_{(u_{0}, u_{1}, 2u_{2})}^{2}$ has a $\Q_{2}$-rational point
if and only if
\[
	(u_{0}, u_{1}, 2u_{2}) \simeq (1, 1, 6), (1, 1, 14), (1, 5, 2), (1, 7, 2), (1, 7, 6).
\]

\end{enumerate}
In particular,,
$X_{\bb{a}}^{2}$ has a $\Q_{2}$-rational point
if and only if
\[
	\bb{a} \simeq (1, 1, 3), (1, 1, 7), (1, 3, 7), (1, 1, 6), (1, 1, 14), (1, 5, 2), (1, 7, 2), (1, 7, 6).
\]
\end{proposition}

\begin{proof}
\begin{enumerate}
\item
We may assume that $u_{0} = 1$ and $u_{1}, u_{2} \in \{ 1, 3, 5, 7 \}$.
Moreover, if $u_{1} = 7$ or $u_{2} = 7$, then $X_{(1, u_{1}, u_{2})}^{2}$ has a rational point.
Therefore, it is sufficient to consider the following six cases.

\begin{enumerate}
\item
If $(u_{1}, u_{2}) = (1, 1)$, then we can check that $X_{(1, 1, 1)}^{2}$ has no  $\Q_{2}$-rational point
by the standard infinite descent argument with a help of modulo 8 calculation.

\item
If $(u_{1}, u_{2}) = (1, 3)$, then $X_{(1, 1, 3)}^{2}$ has a $\Q_{2}$-rational point $[\sqrt{-7} : 2 : 1]$.

\item
If $(u_{1}, u_{2}) = (1, 5)$, then we can check that $X_{(1, 1, 5)}^{2}$ has no  $\Q_{2}$-rational point
by the standard infinite descent argument with a help of modulo 8 calculation.

\item
If $(u_{1}, u_{2}) = (3, 3)$, then
$X_{(1, 3, 3)}^{2}$ is isomorphic to $X_{(1, 1, 3)}^{2}$ over $\Q_{2}$ and has a $\Q_{2}$-rational point.

\item
If $(u_{1}, u_{2}) = (3, 5)$, then
$X_{(1, 3, 5)}^{2}$ is isomorphic to $X_{(1, 7, 3)}^{2}$ over $\Q_{2}$ and has a $\Q_{2}$-rational point.

\item
If $(u_{1}, u_{2}) = (5, 5)$, then
$X_{(1, 5, 5)}^{2}$ is isomorphic to $X_{(1, 1, 5)}^{2}$ over $\Q_{2}$ and has no $\Q_{2}$-rational point.

\end{enumerate}

\item
We may assume that $u_{2} = 1$ and $u_{0}, u_{1} \in \{ 1, 3, 5, 7 \}$.
Note that if $x_{0}x_{1} \equiv 0 \pmod{2}$,
then $x_{0} \equiv x_{1} \equiv 0 \pmod{2}$ and $x_{2} \equiv 0 \pmod{2}$.
Therefore,
it is sufficient to consider rational points such that $x_{0}, x_{1} \in \Z_{2}^{\times}$.

\begin{enumerate}
\item
If $u_{1} \equiv u_{0} \pmod{8}$, i.e., $u_{1} = u_{0}$,
then we have $2u_{0}+2x_{3}^{2} \equiv 0 \pmod{8}$,
which has a $\Z_{2}$-solution only if $u_{0} \equiv 3 \pmod{4}$,
i.e., $(u_{0}, u_{1}) = (3, 3), (7, 7)$.
We can check that
$X_{(3, 3, 2)}^{2}$ (resp. $X_{(7, 7, 2)}^{2}$) has a $\Q_{2}$-rational point
$[x_{0} : x_{1} : x_{2}] = [1 : 3 : \sqrt{-15}]$ (resp. $[1 : 1 : \sqrt{-7}]$).

\item
If $u_{1} \equiv 3u_{0} \pmod{8}$,
then we have $4u_{0}+2x_{3}^{2} \equiv 0 \pmod{8}$,
which is impossible.

\item
If $u_{1} \equiv 5u_{0} \pmod{8}$,
then we have $6u_{0}+2x_{3}^{2} \equiv 0 \pmod{8}$,
which has a solution only if $u_{0} \equiv 1 \pmod{4}$,
i.e., $(u_{0}, u_{1}) = (1, 5), (5, 1)$.
We can check that
$X_{(1, 5, 2)}^{2}$ has a $\Q_{2}$-rational point
$[x_{0} : x_{1} : x_{2}] = [1 : 3 : \sqrt{-23}]$,
and $(5, 1, 2) \simeq (1, 5, 2)$.

\item
If $u_{1} \equiv 7u_{0} \pmod{8}$,
then $X_{(u_{0}, u_{1}, u_{2})}^{2}$ has a $\Q_{2}$-rational point
$[x_{0} : x_{1} : x_{2}] = [1 : \sqrt{-7} : 0]$.
Note that $(1, 7, 14) \simeq (1, 7, 2)$ 
and $(1, 7, 10) \simeq (1, 7, 6)$.
\end{enumerate}
\end{enumerate}
This completes the proof.
\end{proof}

\begin{proposition} [$p = 2$ and $n = 3$] \label{quad_surface}
Let $\bb{a} = (a_{0}, a_{1}, a_{2}, a_{3}) \in \Z_{p}^{\oplus 4} \setminus \Delta_{3}$.
Then, $X_{\bb{a}}^{2}$ has no $\Q_{2}$-rational point
if and only if
\[
	\bb{a} \simeq (1, 1, 1, 1), (1, 1, 5, 5), (1, 1, 2, 2), (1, 1, 10, 10), (1, 3, 2, 6), (1, 3, 10, 14), (1, 5, 6, 14).
\]
\end{proposition}

\begin{proof}
First of all,
note that
we may assume that $a_{0} = 1$, $v_{2}(a_{1}) = 0$, $0 \leq v_{2}(a_{2}), v_{2}(a_{3}) \leq 1$,
and $a_{i}p^{-v_{p}(a_{i})} \in \{ 1, 3, 5, 7 \}$.

\begin{enumerate}
\item
Suppose that $v_{2}(\bb{a}) = (0, 0, 0, 0)$.
If $a_{i}/a_{j} \equiv -1 \pmod{8}$ for some $i, j$,
then $X_{\bb{a}}^{2}$ has a $\Q_{2}$-rational point such that
$(x_{i}, x_{j}) = (1, (-a_{i}/a_{j})^{1/2})$.
Therefore, it is sufficient to consider the cases
$\bb{a} = (1, 1, 1, 1), (1, 1, 1, 3), (1, 1, 1, 5), (1, 1, 3, 3), (1, 1, 5, 5)$.

\begin{enumerate}
\item
Suppose that $\bb{a} = (1, 1, 1, 1)$.
Then, $X_{\bb{a}}^{2}$ has no $\Q_{2}$-rational point.

\item
Suppose that $\bb{a} = (1, 1, 1, 3)$.
Then, $X_{\bb{a}}^{2}$ has a $\Q_{2}$-rational point
$[x_{0} : x_{1} : x_{2} : x_{3}] = [\sqrt{-7} : 2 : 0 : 1]$.

\item
Suppose that $\bb{a} = (1, 1, 1, 5), (1, 1, 3, 3)$.
Then, $X_{\bb{a}}^{2}$ has a $\Q_{2}$-rational point
$[x_{0} : x_{1} : x_{2} : x_{3}] = [\sqrt{-7} : 1 : 1 : 1]$.

\item
Suppose that $\bb{a} = (1, 1, 5, 5)$.
Then, $X_{\bb{a}}^{2}$ has no $\Q_{2}$-rational point.
\end{enumerate}

\item
Suppose that $v_{2}(\bb{a}) = (0, 0, 0, 1)$.
Then, by the proof of \cref{quad_curve} (2),
it is sufficient to consider the cases
$(a_{0}, a_{1}, a_{2}) = (1, 1, 1), (1, 1, 5)$.
Moreover, by \cref{quad_curve} (1),
it is sufficient to consider the cases
$(a_{0}, a_{1}, a_{2}, a_{3}) = (1, 1, 1, 2), (1, 1, 1, 10)$.
In each case,
$X_{\bb{a}}^{2}$ has a $\Q_{2}$-rational point
$[x_{0} : x_{1} : x_{2} : x_{3}] = [\sqrt{-7} : 1 : 2 : 1], [\sqrt{-15} : 1 : 2 : 1]$ respectively.

\item
Suppose that $v_{2}(\bb{a}) = (0, 0, 1, 1)$.
Then, by \cref{quad_curve} (1),
it is sufficient to consider the cases $a_{1} \in \{ 1, 3, 5 \}$.
\begin{enumerate}
\item
Suppose that $a_{1} = 1$.
Then, by \cref{quad_curve} (1),
it is sufficient to consider the cases $a_{2}, a_{3} \in \{ 2, 10 \}$.
If $(a_{2}, a_{3}) = (2, 10)$ (resp. $(10, 2)$),
then $X_{\bb{a}}^{2}$ has a $\Q_{2}$-rational point
$[x_{0} : x_{1} : x_{2} : x_{3}] = [2 : 0 : \sqrt{-7} : 1]$ (resp. $[2 : 0 : 1 : \sqrt{-7}]$).
If $a_{2} = a_{3}$,
then $X_{\bb{a}}^{2}$ has no $\Q_{2}$-rational point.

\item
Suppose that $a_{1} = 3$.
Then, by \cref{quad_curve} (1),
it is sufficient to consider the cases $(a_{2}, a_{3}) = (2, 6), (10, 14)$.
In both cases,
$X_{\bb{a}}^{2}$ has no $\Q_{2}$-rational point.

\item
Suppose that $a_{1} = 5$.
Then, by \cref{quad_curve} (1),
it is sufficient to consider the cases $a_{2}, a_{3} \in \{ 6, 14 \}$.
If $a_{2} = a_{3} = 6$ (resp. $a_{2} = a_{3} = 14$),
then $X_{\bb{a}}^{2}$ has a $\Q_{2}$-rational point
$[x_{0} : x_{1} : x_{2} : x_{3}] = [6 : 0 : \sqrt{-7} : 1]$ (resp. $[14 : 0 : \sqrt{-15} : 1]$).
If $a_{2} \neq a_{3}$,
then $X_{\bb{a}}^{2}$ has no $\Q_{2}$-rational point.
\end{enumerate}
\end{enumerate}

This completes the proof.
\end{proof}

\begin{remark}
In fact, 
the $\Gamma_{2}(3, 2)$-orbits of the $7$ vectors in the statement of \cref{quad_surface} do not intersect each other. 
We can check it by noting that
\begin{itemize}
\item 
each orbit has a representative whose components lie in $\{1,3,5,7,2,6,10,14 \}$, 

\item 
in terms of these representatives, the $\Gamma_{2}(3, 2)$-action is reduced to the action of a finite group $2^{\Z}/2^{2\Z} \times \Z_{2}^{\times}/\Z_2^{\times 2} \times \mathfrak{S}_{3}$. 
\end{itemize} 
\end{remark}

\begin{proposition} [$p = 2$ and $n \geq 4$] \label{quad_general}
Let $n \in \Z_{\geq 4}$ and $\bb{a} = (a_{0}, \dots, a_{n}) \in \Z_{2}^{\oplus n+1} \setminus \Delta_{n}$.
Then, $X_{\bb{a}}^{2}$ has a $\Q_{2}$-rational point.
\end{proposition}

\begin{proof}
By the case (b) in the proof for \cref{quad_surface},
it is sufficient to consider the case
$v_{2}(\bb{a}) = \bb{0}$.
Moreover,
By the case (a) in the proof for $n = 3$,
it is sufficient to consider the cases
$n = 4$ and $\bb{a} = (1, 1, 1, 1, 1), (1, 1, 3, 5, 5)$.
In each case,
$X_{\bb{a}}^{2}$ has a $\Q_{2}$-rational point
$[x_{0} : x_{1} : x_{2} : x_{3} : x_{4}] = [\sqrt{-7} : 2 : 1 : 1 : 1], [0 : 0 : \sqrt{-15} : 3 : 0]$.
\end{proof}

\begin{proof} [Proof of \cref{quadratic_main}]
For $p \neq 2$,
the statement is immediate from \cref{generic}.
For $p = 2$,
the statement is a direct consequence of \cref{quad_curve,quad_surface,quad_general} as follows.

\begin{enumerate}
\item
For $n = 2$, 
we have
\begin{align*}
\rho_{2}(2, 2)
&= \mu_{2} \left( \bb{a} \simeq (1, 1, 3), (1, 1, 7), (1, 3, 7), (1, 1, 6), (1, 1, 14), (1, 5, 2), (1, 7, 2), (1, 7, 6) \right) \\
&= \kappa_{2}(2, 2)\cdot
\frac{3 \cdot 4 + 3 \cdot 4 + 3! \cdot 4}{4^{3}}\cdot
\mu_{2} \left( v_{2}(\bb{a}) = (0, 0, 0), (1, 1, 1) \right)\\
&\quad + \kappa_{2}(2, 2) \cdot \frac{3 \cdot 4 + 3 \cdot 4 + 3! \cdot 4 + 3! \cdot 4 + 3! \cdot 4}{4^{3}} \cdot 
\mu_{2}\left( v_{2}(\bb{a}) = (0, 0, 1), (1, 1, 0) \right) \\
&= \frac{2^{6}}{3^{3}} \cdot \frac{3}{4} \cdot \left( \frac{1}{2^{3}} + \frac{1}{2^{6}} \right)
+ \frac{2^{6}}{3^{3}} \cdot \frac{3}{2} \cdot \left( \frac{1}{2^{4}} + \frac{1}{2^{5}} \right)
= \frac{1}{4} + \frac{1}{3}
= \frac{7}{12}.
\end{align*}
Here, 
in order to obtain the second equality, 
we use, for example, 
\begin{align*}
\{\bb{a}\in\Z_{p}^{\oplus 3} \mid \bb{a} \simeq (1,5,2) \}
= \coprod_{k_{0}, k_{1}, k_{2} \geq 0} (2^{2k_{0}}, 2^{2k_{1}}, 2^{2k_{2}})
\mathfrak{S}_{3}A
\end{align*}
and 
\[
\mu_{2}(A) = \frac{3! \cdot \#\left(\Z_{2}/\Z_{2}^{\times 2}\right)}{\#\left(\Z_{2}/\Z_{2}^{\times 2}\right)^{3}}\mu_{2}(v_{2}(\bb{a}) = (0,0,1), (1,1,0)),
\]
where
\[
A := \left\{ \bb{a} \in \Z_{2}^{\oplus 3} \setmid \text{$\bb{a} = (ut_{0}, 5ut_{1}, 2ut_{2}), (2ut_{0}, 10ut_{1}, ut_{2})$ with $t_{i} \in \Z_{2}^{\times 2}, u \in \Z_{2}^{\times}$} \right\},
\]

\item
Similarly, 
for $n = 3$,
we have
\begin{align*}
	\rho_{2}
	&= 1 - \mu_{2}\left( \bb{a} \simeq (1, 1, 1, 1), (1, 1, 5, 5), (1, 1, 2, 2), (1, 1, 10, 10), (1, 3, 2, 6), (1, 3, 10, 14), (1, 5, 6, 14) \right) \\
	&= 1 - \kappa_{2}(3, 2) \cdot \frac{4 + \binom{4}{2} \cdot 2}{4^{4}} \cdot  
		\mu_{2}\left( v_{2}(\bb{a}) = (0, 0, 0, 0), (1, 1, 1, 1) \right) \\
	&\quad - \kappa_{2}(3, 2) \cdot \frac{\binom{4}{2} \cdot 4 + \binom{4}{2} \cdot 4 + 4! \cdot 2 + 4! \cdot 2 + 4! \cdot 2}
	{4^{4}} \cdot 
		 \mu\left( v_{2}(\bb{a}) = (0, 0, 1, 1) \right) \\
	&= 1 - \frac{2^{8}}{3^{4}} \cdot \frac{4}{4^{3}} \cdot \left( \frac{1}{2^{4}} + \frac{1}{2^{8}} \right)
		- \frac{2^{8}}{3^{4}} \cdot \frac{48}{4^{3}} \cdot \frac{1}{2^{6}}
			= 1 - \frac{17 + 48}{2^{4} \cdot 3^{4}}
				= \frac{1231}{1296}.
\end{align*}

\item
For $n = 4$,
the statement is obvious from \cref{quad_general}.
\end{enumerate}
\end{proof}

\subsection{The case of $k = 3$ and $p = 3$}

\begin{proposition} [$p = 3$ and $n = 2$] \label{cubic_curve}
	Let $u_{0}, u_{1}, u_{2} \in \mathbb{Z}_{3}^{\times}$.
	\begin{enumerate}
		\item
		$X_{(u_{0}, 3u_{1}, 9u_{2})}^{3}$ has no $\mathbb{Q}_{3}$-rational point.
		
		\item
		$X_{(u_{0}, u_{1}, 9u_{2})}^{3}$ has a $\mathbb{Q}_{3}$-rational point
		if and only if $u_{0} \equiv \pm u_{1} \pmod{9}$.
		
		\item
		$X_{(u_{0}, u_{1}, 3u_{2})}^{3}$  has a $\mathbb{Q}_{3}$-rational point.
		
		\item
		$X_{(u_{0}, u_{1}, u_{2})}^{3}$  has a $\mathbb{Q}_{3}$-rational point
		if and only if $\{ \pm u_{0}, \pm u_{1}, \pm u_{2} \} \not\equiv \{ \pm1, \pm2, \pm4 \} \pmod{9}$
	\end{enumerate}
\end{proposition}

\begin{proof}
	First of all,
	note that since $\mathbb{Z}_{3}^{\times 3} = \pm1+9\mathbb{Z}_{3}$,
	we may assume that $u_{0}, u_{1}, u_{2} \in \{ 1, 2, 4 \}$
	by replacing $x_{i}$ to $w_{i}x_{i}$ with some $w_{i} \in \mathbb{Z}_{3}^{\times}$ if necessary.
	
	\begin{enumerate}
		\item[(1)(2)]
		These are immediate by an infinite descent with a help of modulo 9 reduction.
		
		\item[(3)]
		If $u_{0} = u_{1}$,
		then our curve has a $\mathbb{Q}_{3}$-rational point such that $x_{0} = -x_{1}$ and $x_{2} = 0$.
		Therefore, it is sufficient to prove that
		for every $(u_{0}, u_{1}) = (1, 2), (1, 4), (2, 4)$
		\[
		u_{0}x_{0}^{3}+u_{1}x_{1}^{3}+3 \cdot 1^{3} = 0
		\]
		has a $\mathbb{Z}_{3}$-solution,
		for instance, $(x_{0}, x_{1}) = (-1, -1), (1, -1), ((-7/2)^{1/3}, 1)$.
		
		\item[(4)]
		By the above argument,
		it is sufficient to prove that
		if $(u_{0}, u_{1}, u_{2}) = (1, 2, 4)$,
		then $X_{(1, 2, 4)}^{3}$ has no $\mathbb{Q}_{3}$-rational point
		(by contradiction as follows).
		Suppose that $X_{(1, 2, 4)}^{3}$ has a $\mathbb{Q}_{3}$-rational point $[x_{0} : x_{1} : x_{2}]$
		with $x_{i} \in \mathbb{Z}_{3}$.
		Then, the congruence
		\[
		x_{0}^{3}+2x_{1}^{3}+4x_{2}^{3} \equiv 0 \pmod{9}
		\]
		implies $x_{0} \equiv x_{1} \equiv x_{2} \equiv 0 \pmod{3}$,
		which is a contradiction by the infinite descent.
	\end{enumerate}
	This completes the proof.
\end{proof}

We define an equivalence relation $\sim$ on the group $\image v_{3} = \Z^{\oplus n+1}$ as the induced equivalence relation by $\simeq$ on $(\Q_{3}^{\times})^{\oplus n+1}$ (cf. \cite[\S 2.2.1]{BBL}).

\begin{proposition} [$p = 3$ and $n \geq 3$] \label{cubic_general}
Let $n \in \Z_{\geq 3}$ and $\bb{a} = (a_{0}, \dots, a_{n}) \in \Z_{3}^{\oplus n+1} \setminus \Delta_{n}$. 
\begin{enumerate}
\item
Suppose that $n = 3$.
Then, $X_{\bb{a}}^{k}$ has a $\mathbb{Q}_{3}$-rational point
if and only if one of the following conditions hold:
\begin{itemize}
\item
$v_{3}(\bb{a}) \sim (0, 0, 0, 0)$, $(0, 0, 0, 1)$, $(0, 0, 1, 1)$, or $(0, 0, 1, 2)$.

\item
$v_{3}(\bb{a}) \sim (0, 0, 0, 2)$,
and if one normalizes $v_{3}(\ba)$ so that \\
$v_{3}(\bb{a}) = (0, 0, 0, 2)$,
then $\{ \pm a_{0}, \pm a_{1}, \pm a_{2} \} \not\equiv \{ \pm 1, \pm 2, \pm 4 \} \pmod{9}$.
\end{itemize}

\item
Suppose that $n \geq 4$.
Then, $X_{\bb{a}}^{k}$ has a $\mathbb{Q}_{3}$-rational point.
\end{enumerate}
\end{proposition}

\begin{proof}
\begin{enumerate}
	\item
	For the detail of the case $n = 3$, see \cite[\S 2.1.2]{BBL}. 
		
	\item
	For $n \geq 4$, it is sufficient to consider the case
	$n = 4$ and $v_{3}(\bb{a}) = (0, 0, 0, 2, 2)$.
	In this case,
	$X_{(a_{0}, a_{3}, a_{4})}^{3} \subset X_{\bb{a}}^{3}$
	has a $\mathbb{Q}_{3}$-rational point by \cref{cubic_curve} (3).
\end{enumerate}
	
This completes the proof.
\end{proof}

\begin{proof} [Proof of \cref{cubic_main}]

We can prove it in a similar manner as the proof of \cref{quadratic_main}. 

\begin{enumerate}

\item
Suppose that $n = 2$. 
By \cref{generic},
it is sufficient to consider the case $p = 3$.
By \cref{cubic_curve}, we have 
\begin{align*}
	\rho_{3}(2, 3) &= 1
	- \mu_{3}\left( \bb{a} \simeq (1,2,4) \right)\\
	&\quad - \mu_{3}\left( \bb{a} \simeq (u_{0}, tu_{0}, 3^{2}u_{1}) \ \text{with} \ u_{0}, u_{1} \in \pintt{3}, t\in \pintt{3} \setminus \Z_{3}^{\times 3} \right)\\
	&\quad - \mu_{3}\left( v_{3}(\bb{a}) \sim (0,1,2) \right)\\
	&= \displaystyle 1 
	- \kappa_{3}(2, 3) \cdot \frac{3!\cdot 1}{3^3} \cdot \mu_{3}\left( v_{3}(\bb{a}) = (0,0,0), (1,1,1), (2,2,2) \right)\\
	&\quad - \kappa_{3}(2, 3) \cdot \frac{3\cdot 3^2\cdot 2}{3^3} \cdot \mu_{3}\left( v_{3}(\bb{a}) = (0,0,2), (1,1,0) \right)\\
	&\quad - \kappa_{3}(2, 3) \cdot \frac{3!\cdot 3^3}{3^3} \cdot \mu_{3}\left( v_{3}(\bb{a}) = (0,1,2) \right)\\ 
	&= \frac{13831}{19773}. 
\end{align*}

\item
For the detail of the case $n = 3$,
see \cite[\S 2.1]{BBL}.

\item
For $n \geq 4$,
the statement is an immediate consequence of \cref{generic,cubic_general}.
\end{enumerate}
In conclusion, we proved the theorem.
\end{proof}

\subsection{Proof of \cref{applicaiton}}

\begin{proof}[Proof of \cref{applicaiton}]
Under the assumption with \cite[Theorem 1.3]{BBL} and \cite[Theorem 1.4]{Browning},
we have $\rho(n, k) = \rho_{\loc}(n, k) = \prod_{v : \text{place}}\rho_{v}(n, k)$. 
Here, note that 
the Hasse principle holds for the case $k=2$ (resp.\ $k=3$ and $n\geq 9$) due to \cite[p.48, Theorem 8]{Serre_course} (resp.\ \cite[Theorem]{Hooley_I}).
We can estimate the last infinite product by using the Riemann zeta function $\zeta(s) = \prod_{p : \text{prime}} (1 - p^{-s})^{-1}$ ($s \in \R_ {>1}$).
For example,
if $n = k = 3$,
we obtain the following inequalities
\[
\zeta(2)^{-4}\prod_{p < 10^{6}}(1 - p^{-2})^{-4}\prod_{p < 10^{6}}\rho_{p}(3,3)
<\prod_{p : \text{prime}}\rho_{p}(3, 3)
<\prod_{p < 10^{6}}\rho_{p}(3,3),
\]  
which give the desired approximation.
\end{proof}

\begin{remark}
In \cite{BBL}, Bright, Browning and Loughran obtained the formulas of $\rho_{p}(3,3)$ ($\sigma_{p}$ in the notation of \cite{BBL}) correctly.
These formulas give the approximation $\rho(3,3)=\rho_{\loc}(3,3) = 0.8964\ldots$.
However, they stated that  $\rho_{\loc}(3,3)$ ($\sigma$ in the notation of \cite{BBL}) equals $0.8605\ldots$, which is incorrect.
\end{remark}

\begin{remark}\label{the case k>3} 
For $k \geq 4$ and $n \geq 3k + 2$, 
\cite[Theorem 1.3]{Brudem-Dietmann} implies that 
\[
	\rho_{loc}(n, k) - \rho(n, k) 
	\leq \lim_{H \to \infty} \frac{cH^{n+1-\theta}}{(2H+1)^{n+1}}
	= 0  
\]
for some $c, \theta \in\R_{> 0}$, hence we obtain $\rho(n, k) = \rho_{\loc}(n, k)$ in this case too. 
\end{remark}

\section{Concluding remarks}

Recall that an algebraic variety defined over $\Q$ is said to be $\Q$-rational
if it is birationally equivalent to $\Proj{n}$ over $\Q$ for some $n$.
Set
\begin{align*}
\delta(n, k)
:= \lim_{H \to \infty} \frac{\left\{ \bb{a} \in \Z^{\oplus n+1} \setmid
	\text{$\lvert \bb{a} \rvert < H$ and $X_{\bb{a}}^{k}$ is $\Q$-rational} \right\}}
{\left\{ \bb{a} \in \Z^{\oplus n+1} \setmid \lvert \bb{a} \rvert < H \right\}}. 
\end{align*}

\begin{proposition} \label{dense_equal} 
\[
	\delta(n, 2) = \rho(n, 2) = \begin{cases}
	0 &\text{if $n = 2$,}\\
	0.8268\ldots  &\text{if $n = 3$,} \\
	1 &\text{if $n\geq 4$.}
	\end{cases} 
\]
\end{proposition}
 
\Cref{dense_equal} follows immediately if we apply the following proposition for nonsingular quadratic hypersurfaces.  

\begin{proposition} \label{equivalence}
	Let $Q \subset \mathbb{P}^{n}$ be a quadratic hypersurface defined over $\Q$ and $n \in \Z_{\geq 1}$.
	Then, 
	$Q$ is $\Q$-rational if and only if it has a non-singular $\Q$-rational point.
\end{proposition}

\begin{proof}
	The only if part is trivial.
	We prove the if part.
	Indeed, the given non-singular $\Q$-rational point has an open neighborhood
	isomorphic to an affine quadratic hypersurface $Q' \subset \A^{n}$ passing through the origin $O = (0, ..., 0)$.
	Then,
	we can take a Zariski dense subset $U$ of $\A^{n}$ so that for every $\Q$-rational point $A$ on $U$ the line $OA$ intersects with $Q'\setminus O$ exactly once. 
	This induces a birational (i.e., generically 1 : 1 and dominant) map
	of $\mathbb{P}^{n-1}$ to $Q'$,
	hence $Q$ itself is $\Q$-rational.
\end{proof}

In a similar manner, we can prove a $v$-adic version of \cref{dense_equal}. 
More precisely,
if we set   
$\delta_{v}(n, 2) := \mu_{v}\left( \text{$X_{\bb{a}}^{2}$ is $\Q_{v}$-rational} \right)$, 
then we obtain $\delta_{v}(n, 2) = \rho_{v}(n, 2)$ for every place $v$.
Therefore, by combining it with \cref{dense_equal}, we obtain the product formula
\begin{align*}\label{productformula}
\delta(n, 2) = \prod_{\text{$v$ : place}}\delta_{v}(n, 2). 
\end{align*}
Moreover, 
the whole argument works also for the family of all quadratic hypersurfaces of $\PP^{n}$ for every fixed $n$ 
(cf. \cite{BCFJK}).
It is a natural question whether or not similar product formulas hold for other families of geometrically rational algebraic varieties.
However,
as far as the authors know, there is no reference answering this question even for all or diagonal cubic surfaces.

On the other hand,
if we replace ``$\Q$-rational" to ``$\Q$-unirational",
then similar product formulas hold for the families of all or diagonal cubic hypersurfaces of $\PP^{n}$ ($n \geq 3$) 
(cf. \cite[Theorem 1.2]{Kollar}, see also \cite[Remark 2.3.1]{CSS}). 
The latter argument works also for the quartic del Pezzo surfaces
defined by 
\[
\begin{cases}
x_{0}x_{1} = x_{2}x_{3}\\
\displaystyle \sum_{i = 0}^{4} a_{i}x_{i} = 0 \quad 
\end{cases}
\]
with $a_{i}\in\Q$ such that $ \prod_{i = 0}^{4}a_{i}(a_{0}a_{1} - a_{2}a_{3}) \neq 0$ (cf.\ \cite{Mitankin-Salgado} and \cite[Theorem 29.4 and Theorem 30.1]{Manin_2nd}).

\section*{Acknowledgements.}

First,
the authors greatly thank their supervisor Prof.\ Ken-ichi Bannai for his careful reading of manuscript and giving many helpful comments.
The authors express their sincere gratitude to Kazuki Yamada and Shuji Yamamoto for their valuable comments.
The authors are also grateful to Timothy Browning for his suggestion on the description of \cref{applicaiton} and the contents of \cref{the case k>3}. 
The authors owe the proof of \cref{applicaiton} to Mathematica.

\end{document}